\documentclass[11pt,leqno]{amsart}
\usepackage{amscd}
\usepackage{color}
\usepackage{amssymb}
\usepackage{amsfonts}
\usepackage{latexsym}
\usepackage{verbatim}
\usepackage{enumitem}

\theoremstyle{plain}
\newtheorem{theorem}{Theorem}[section]
\newtheorem{definition}[theorem]{Definition}
\newtheorem{lemma}[theorem]{Lemma}
\newtheorem{prop}[theorem]{Proposition}
\newtheorem{cor}[theorem]{Corollary}
\newtheorem{rem}[theorem]{Remark}

\newcommand{\lb}{[\cdot,\cdot]}

\newcommand{\ad}{{\operatorname{ad}}}
\newcommand{\mm}{{\operatorname{M}}}
\newcommand{\la}{\langle}
\newcommand{\ra}{\rangle}
\newcommand{\ip}{\la \cdot, \cdot \ra}
\newcommand{\tr}{\operatorname{tr}}

\newcommand{\g}{\mathfrak{g}}
\newcommand{\h}{\mathfrak{h}}

\newcommand{\n}{\mathfrak{n}}

\newcommand{\s}{\mathfrak{s}}
\renewcommand{\r}{\mathfrak{r}}
\renewcommand{\a}{\mathfrak{a}}

\newcommand\C{{\mathbb C}}
\newcommand\R{{\mathbb R}}

\newcommand{\N}{\nabla}

\title[Expanding solitons to the HCF on complex Lie groups]{Expanding solitons to the Hermitian curvature flow on complex Lie groups}
\thanks{This work was supported by G.N.S.A.G.A. of I.N.d.A.M} 

\author{Mattia Pujia}
\address{Dipartimento di Matematica G. Peano, Universit\`a di Torino, Via Carlo Alberto 10, 10123 Torino, Italy}
\email{mattia.pujia@unito.it}

\date{\today}

\begin{document}

\begin{abstract}We investigate the algebraic structure of complex Lie groups equipped with left-invariant metrics which are expanding semi-algebraic solitons to the Hermitian curvature flow (HCF). We show that the Lie algebras of such Lie groups decompose in the semidirect product of a reductive Lie subalgebra with their nilradicals. Furthermore, we give a structural result concerning expanding semi-algebraic solitons on complex Lie groups. It turns out that the restriction of the soliton metric to the nilradical is also an expanding algebraic soliton and we explain how to construct expanding solitons on complex Lie groups starting from expanding solitons on their nilradicals. 
\end{abstract}

\maketitle

\section{Introduction} 
In 2011 Streets and Tian introduced a new flow of Hermitian metrics called {\em Hermitian curvature flow} (HCF) \cite{ST2}. The flow evolves an initial metric in the direction of a Ricci-type tensor of the Chern curvature modified with some first order terms in the torsion. The defining equation is strictly parabolic and when the initial metric is K\"ahler the HCF reduces to the K\"ahler-Ricci flow.
\medskip

The flow is defined as follow. Let $(M,g)$ be a Hermitian manifold of complex dimension $n$, with Chern connection $\N$ and Chern curvature tensor $\Omega$. Let $S_{i\bar j}=g^{l\bar k}\Omega_{l\bar k i \bar j}$ be the $(1,1)$-tensor obtained by contracting $\Omega$ in the first two entries and let
\begin{equation*}\label{eqn_K}
K(g):=S(g)-Q(g)\,,
\end{equation*}
where $Q$ is a $(1,1)$-tensor quadratic in the torsion components (see \cite{ST2} for the precise definition of $Q$). Then HCF is defined by 
\begin{equation}\label{HCF}
\partial_tg_t=-K(g_t)\,,\qquad g_{|t=0}=g_0,
\end{equation}
where $g_0$ is an initial Hermitian metric on $M$. In \cite{ST2} the tensor $Q$ is chosen in order to make the flow satisfying a gradient-type equation. Nevertheless, since $Q(g)$ contains only first order terms in $g$, different choices of $Q$ do not affect the parabolicity of the flow and these lead to different interesting flows (see e.g. \cite{S}, \cite{S2}, \cite{Ssolitons}, \cite{ST}, \cite{ST3}, \cite{ST4}, \cite{yuri}, \cite{yuri2}, and the references therein). 

\medskip
In this paper we focus on {\em soliton} solutions to the HCF, i.e. Hermitian metrics satisfying  
\begin{equation}\label{eqn_soliton}
	K(g)=c g+\mathcal{L}_X g,
\end{equation}
for some $c \in \R$ and a complete holomorphic vector field $X$, where $\mathcal L$ denotes the Lie derivative. By definition, $K$ is both scale invariant and diffeomorphisms equivariant. Therefore, if $g_0$ satisfies \eqref{eqn_soliton} then the solution to \eqref{HCF} safitisfies $g_t = s(t) \, \varphi_t^* g_0$, where $s(t) > 0$ and $\varphi_t : M \to M$ are respectively a smooth scaling function and a one-parameter family of biholomorphisms.

\medskip
In complex Lie groups context it is quite natural to focus on \emph{semi-algebraic} solitons, which are left-invariant metrics $g_0$ such that the solutions to \eqref{HCF} have the form $g_t = s(t) \, \varphi_t^* g_0$ and $\varphi_t$ is a Lie group automorphism for every $t$. If further $\partial_t \varphi_t^{-1}|_{t=0}$ is $g$-self-adjoint, then $g$ is called an {\em algebraic} soliton. Semi-algebraic and algebraic solitons of other flows have been studied in \cite{ALpluri}, \cite{Jab15}, \cite{lauret1}, \cite{Lau16}, \cite{Lau11}.

\medskip
Now we state the main result of the paper. Let $(G,g)$ be a complex Lie group equipped with a left-invariant Hermitian metric and consider the orthogonal splitting of its Lie algebra $\g$ in
$$\g=\r\oplus\n\,,$$
where $\n$ is the nilradical of $\g$. Denote by $g_\n$ the pull-back of $g$ to the Lie group $N$ of $\n$. Then we have

\begin{theorem}\label{main_theo} The metric $g$ is an expanding (i.e. $c<0$) semi-algebraic soliton to HCF if and only if $g_\n$ is an expanding algebraic soliton to HCF on $N$, $\r$ is a reductive Lie subalgebra, $\sum[{\rm ad}_{r_i}|_\n,{\rm ad}_{\bar r_i}^t|_\n]=0$ for any unitary basis $\{r_i\}$ of $\r$, and 
$$K(g_\r)(X,\bar Y)=c g_\r(X,\bar Y)+\frac 1{2}\tr (\ad_X|_\n \ad_{\bar Y}^t|_\n)-\frac 12 \tr \ad_X\cdot\tr\ad_{\bar Y}\,,$$
for any $X,Y\in \r$, where $g_\r$ is the pull-back of $g$ to the Lie group of $\r$.
\end{theorem}

\medskip 
Note that if $G$ is unimodular, then the expression of $K(g_\r)$ in Theorem \ref{main_theo} simplifies to
$$K(g_\r)(X,\bar Y)=c g_\r(X,\bar Y)+\frac 1{2}\tr (\ad_X|_\n \ad_{\bar Y}^t|_\n)\,.$$
Our interest in {\em expanding} algebraic solitons on complex unimodular Lie groups comes from \cite{LPV}, where it is proved that expanding algebraic solitons on such Lie groups are limit points to the \emph{normalized} HCF. Indeed, when $(G,g_0)$ is a complex unimodular Lie group equipped with a left-invariant metric, the solution $g_t$ to the HCF starting from $g_0$ is defined for every positive $t$ and ${(G,(1+t)^{-1}g_t)}$ converges in Cheeger-Gromov sense to $(\bar G,\bar g)$, where $\bar G$ is a complex unimodular Lie group and $\bar g$ is an algebraic soliton.

Next we observe that in the solvable case we can improve Theorem \ref{main_theo} by giving an explicit description of $g_\r$.

\begin{cor}\label{main_theo1} Assume $G$ unimodular and solvable. Then, $g$ is an expanding algebraic soliton to HCF if and only if $g_\n$ is an expanding algebraic soliton to HCF on $N$, the Lie group $G$ is standard (i.e. $\g=\r\oplus\n$ with $\r$ abelian), $\sum[{\rm ad}_{r_i}|_\n,{\rm ad}_{\bar r_i}^t|_\n]=0$ for any unitary basis $\{r_i\}$ of $\r$,, and 
$$g_\r(X,\bar Y)=-\frac 1{2c}\tr (\ad_X|_\n \ad_{\bar Y}^t|_\n)\,,$$
for any $X,Y\in \r$.
\end{cor}

\medskip
The proof of Theorem \ref{main_theo} is mainly based on {\em real geometric invariant theory} (GIT), in the same fashion as in \cite{Lau10}.

\medskip
Similar results, concerning the Ricci flow on different homogeneous spaces, can be found in \cite{LafLauJDG} and \cite{Lau11}. However, as pointed out by Lafuente and Lauret in \cite{LafLauJDG}, for the Ricci flow there is a limitation given by Alekseevskii’s conjecture. Indeed, if Alekseevskii’s conjecture is confirmed, then any Ricci flow expanding algebraic soliton $(G/H,g)$ should be diffeomorphic to an Euclidean space \cite{Jab15} and thus, accordingly, only solvmanifolds could admit expanding algebraic solitons to the Ricci flow. In the HCF case such a limitation does not exist. As shown in \cite{LPV}, also semisimple complex Lie groups admit soliton metrics. Specifically, a complex Lie group $G$ admits a left-invariant static Hermitian metric,  i.e. a metric satisfying the Einstein-type equation 
\begin{equation*}\label{eqn_static}
K(g)=c\, g,
\end{equation*}
for some constant $c \in \mathbb{R}$, if and only if the group is semisimple and the static metric is the `{\em canonical metric}' induced by the Killing form. Hence, we have a wider set of expanding algebraic solitons for the HCF, with algebraic structures completely classified by Theorem \ref{main_theo} in the case of complex Lie groups.

\medskip
The paper is organized as follows. In Section 2 we recall main results about HCF on complex Lie groups and GIT on Lie groups. In Section 3 we prove Theorem \ref{main_theo} and its corollary. Finally, in Section 4 we apply Corollary \ref{main_theo1} to construct explicit examples of  expanding algebraic solitons to HCF on 4-dimensional solvable complex  unimodular Lie groups. 

\bigskip
\noindent {\bf Notation and conventions.} By a \emph{complex Lie group} we mean a Lie group endowed with a bi-invariant complex structure (i.e. the multiplication is a holomorphic map).

\bigskip
\noindent {\bf Acknowledgments.} The research of the present paper was originated by some conversations with Jorge Lauret, during a visiting period of the author at FaMAF (Cordoba). The author is very grateful to Lauret for many useful suggestions and insights on the problems studied in the paper, and to Luigi Vezzoni for his comments on a preliminary version of the paper. The author would like to thank the referee for him/her constructive comments, which helped to improve the paper.

\section{HCF and GIT results}
In this section we recall some results on the HCF on complex Lie groups and GIT which will be useful in the sequel. 

\subsection{HCF on complex Lie groups}
The following proposition characterizes the HCF tensor on complex Lie groups.

\begin{prop}\cite{LPV} Let $G$ be a complex Lie group equipped with a left-invariant Hermitian metric $g$. Then
\begin{equation}\label{K_tens}
K(g)={\rm Ric}^{1,1} + \hat Q \,,
\end{equation}
where ${\rm Ric}^{1,1}$ is the $(1,1)$-part of the Ricci tensor of $g$ and
$$\hat Q(Z,\bar W) := \frac 12 \tr  \ad_Z \cdot \tr \ad_{\bar W}\,.$$
Here, $Z,W$ are left-invariant vector fields of type (1,0).
\end{prop}

It is well known (see e.g. \cite{Besse}) that the Ricci tensor of a left-invariant metric $g$ on a Lie group $G$ can be written as
$${\rm Ric}= \mm-\frac 12 {\rm B}-S(\ad_H)\,,$$
where, for any $X,Y$ in the Lie algebra $(\g,\mu)$ of $G$,
\begin{equation}\label{M_tens}
{\rm M}(X,Y)=-\frac 12  \, \sum_k g( \mu(X,X_k),\mu(Y,X_k))+\frac 14  \sum_{k,j}\, g( \mu(X_k,X_j),X) g(\mu(X_k,X_j),Y)\,.
\end{equation}
Here $\{X_r\}$ denotes an orthonormal basis of $g$; $H$ is the {\em mean curvature vector} given by the relation $g(H,X)= \tr \ad_X$, for any $X\in\g$, and 
$$S(\ad_H)(X,Y) = \frac 12 (g(\mu(H,X),Y)+g(\mu(H,Y),X))\,;$$
${\rm B}(X,Y)=\tr (\ad_X\ad_Y)$ is the Killing form of $\g$. If further $G$ is a complex Lie group and $g$ is a left-invariant Hermitian metric $g$, then the (1,1)-part of the Ricci tensor satisfies
$${\rm Ric}^{1,1} = \mm - S(\ad_H)\,.$$  
Finally, when the Lie group $G$ is unimodular ${\rm Ric}^{1,1}=\mm$, since the $S(\ad_H)$-term vanishes.

\medskip
Although our goal is to study solutions to the HCF on complex Lie groups, our results hold true for left-invariant solutions $g_t$ to the ${\sf K}$-\emph{flow}
\begin{equation}\label{M_flow}
\partial_tg_t=-{\sf K}(g_t)\,,\qquad g_{|t=0}=g_0\,,
\end{equation} 
on Lie groups, where 
\begin{equation}\label{K-ten} {\sf K}(g) := \mm - S(\ad_H) +\hat Q\end{equation}
and $\hat Q(X,Y):= \frac 12 \tr \ad_X\cdot \tr \ad_Y$. From now on we focus on this more general setting and we obtain the results stated in the introduction as special cases. 
\begin{definition}\label{def_sol}
A left-invariant metric $g$ on a Lie group $G$ is a {\rm semi-algebraic ${\sf K}$-soliton} if its ${\sf K}$-tensor \eqref{K-ten} satisfies
\begin{equation*}
	{\sf K}(g) = c\, g + \frac 12 \left(g( D \cdot, \cdot) + g(\cdot, D \cdot)\right), \qquad c \in \mathbb{R}, \quad D \in {\rm Der}(\g).
\end{equation*}
If further $D^t\in{\rm Der}(\g)$, then the soliton is \emph{algebraic}.
\end{definition}
Note that, we can regard ${\sf K}(g)$ as an endomorphism
$${\sf K}_g:\g\to\g$$
of the Lie algebra of $G$ via $g({\sf K}_g\cdot,\cdot)={\sf K}(g)(\cdot,\cdot)$. Thus, the semi-algebraic ${\sf K}$-soliton condition can be written in terms of $\sf K_g$ as 
$${\sf K}_g = c\,I+\frac 12(D+D^t), \qquad c \in \mathbb{R}, \quad D \in {\rm Der}(\g)\,.$$

\begin{rem}\rm Every semi-algebraic ${\sf K}$-soliton is a soliton in the usual sense. Indeed, if $G$ is a simply-connected Lie group, then the solution to \eqref{M_flow} starting from a semi-algebraic ${\sf K}$-soliton $g$ is $g_t=(-c\, t+1)\varphi_t^\ast g$, where $\varphi_t\in {\rm Aut}(G)$ is the unique automorphism such that $d\varphi_t|_e=e^{-tD/2}\in{\rm Aut}(\g)$.
\end{rem}

\subsection{GIT on Lie groups}\label{GIT}
Let $N$ be a Lie group with Lie algebra $(\n,\mu_0)$. The Lie  bracket of $\n$ is an element of the \emph{variety of Lie algebras} (see e.g. \cite{lauret3}, \cite{Lau13}, \cite{lauret1})
\[
\mathcal C=\left\{\mu\in\Lambda^2\n^*\otimes \n : \mu \mbox{ satisfies the Jacobi identity}\right\}. 
\] 
By changing $\mu\in\mathcal C$ we obtain all the possible structures of Lie algebra on the vector space $\n$. The Lie group ${\rm GL}(\n)$ acts canonically on $${\rm V}:=\Lambda^2 \n^* \otimes \n$$ by 
\begin{equation}\label{action}
	A\cdot \mu(\cdot,\cdot)=A\,\mu(A^{-1}\cdot,A^{-1}\cdot)\,.
\end{equation}
The action induces the Lie algebra representation $\pi : {\rm End}(\n) \to {\rm End}( \Lambda^2 \n^* \otimes \n)$ given by
\begin{equation}\label{delta}
 \big(\pi(E)\mu \big) (X,Y) :=  E(\mu(X,Y)) - \mu(E(X),Y) - \mu(X,E (Y))\,, \quad X, Y\in \n, \quad E\in {\rm End}(\n)\,,
\end{equation}
and it satisfies $\pi(D)\mu=0$, for any derivation $D\in {\rm Der}(\n)$.

Now we fix a inner product $g$ on $\g$ and for $A,B\in{\rm End}(\n)$ we denote by
\[
	\la  A, B \ra := {\rm tr}\, A B^t\
\]
the scalar product induced on ${\rm End}(\n)$, where the transpose is given with respect to $g$. In order to simplify the notation we still denote with $\ip$ the scalar product induced on $\Lambda^2 \n^* \otimes \n$. The pair $(\mu_0,g)$ induces a tensor $\mm$ via \eqref{M_tens}. Using the metric, we can regard $\mm$ as an endomorphism $$\mm_g:\n\to\n\,.$$ By fixing $g$ and changing $\mu$ we obtain a different $\mm$-endomorphism in ${\rm End}(\n)$ which we denote by $\mm_\mu$. In this way, we have a map from $C$ to ${\rm End}(\n)$, $\mu\mapsto\mm_\mu$. Note that by definition 
$$\mm_{\mu_0}=\mm_{g}\,.$$

\begin{prop}\cite{lauret06} The map
\[
\qquad\mu \mapsto \frac{4}{\Vert \mu \Vert^2} \, \mm_\mu
\] 
from $\Lambda^2 \n^* \otimes \n \setminus \{ 0\}$ to ${\rm End}(\n)$ is a moment map, in the sense of GIT, i.e. 
\begin{equation}\label{eqn_mm}
	\la \mm_\mu, E \ra = \tfrac14 \, \la  \pi(E)\mu, \mu \ra\,,
\end{equation}
for any $E\in {\rm End}(\n)$ and $\mu\in\Lambda^2 \n^* \otimes \n \setminus \{ 0\}$.
\end{prop}

\medskip
Next, we recall a stratification theorem involving $\rm V$ proved in \cite{Lau10}. Fix a basis in $\n$ and for any element $\mu\in\rm V$ denotes with $\mu_{ij}^k$ its components. Moreover, let
$$\mathcal N := \lbrace \mu\in\mathcal C : \mu \text{ is nilpotent}\rbrace$$
be the {\em variety of nilpotent Lie algebras},
$$\mathfrak t^+:=\{\beta={\rm diag}(a_1,\ldots,a_n)\in \mathfrak t: a_1\leq\ldots\leq a_n\}$$
and  $\alpha_{ij}^k:= E_{kk}-E_{ii}-E_{jj}$, where $E_{ij}$ is the zero matrix with 1 in the $ij$-entry. Here, $\mathfrak t$ denotes the maximal torus algebra in $\mathfrak{gl}_n(\R)$ given by the $n\times n$ diagonal matrices.

\begin{theorem}\label{prop_rest}\cite{Lau10, LW} 
There exists a finite subset $\mathcal{B}\subset\mathfrak t^+$, such that every $\beta\in \mathcal{B}$ satisfies ${\rm tr}\,\beta=-1$ and 
$${\rm V}\backslash\{0\}=\bigcup_{\beta\in \mathcal{B}}S_\beta \quad \mbox{(disjoint union)}\,,$$
where $\{S_\beta\}_{\beta\in \mathcal{B}}$ is a family of ${\rm GL}_n(\R)$-invariant subset of $\rm V$. Given $\mu \in \mathcal{S_\beta}$ 
\begin{equation}
\beta+\lVert\beta\rVert^2 I\quad\mbox{is positive definite } \forall \beta \in \mathcal{B}\,\text{such that}\,\,\mathcal{S_\beta}\cap\mathcal{N}\neq\emptyset\,,
\end{equation}
\begin{equation}\label{cond_brac}\langle [\beta,D],D\rangle\geq 0\,,\quad \forall D\in{\rm Der}(\mu)\quad \text{(equality holds $\Leftrightarrow$ $[\beta,D]=0$)}
\end{equation}
and
\begin{equation}\label{cond_m} \lVert\beta\rVert \leq \frac 4{\lVert\mu\rVert^2} \lVert \mm_\mu\rVert\,.
\end{equation}
The equality in Moreover, if $\mu\in\mathcal{S_\beta}$ satisfies 
\begin{equation}
\label{cond_rest}{\rm min}\{\langle\beta,\alpha_{ij}^k\rangle:\mu_{ij}^k\neq 0\}=\lVert\beta\rVert^2\,,
\end{equation}
then
\begin{equation}
\label{rest_2} \langle\pi(\beta+\lVert\beta\rVert^2I)\mu,\mu\rangle\geq 0
\end{equation}
and 
\begin{equation}
\label{rest_1} {\rm tr}\,\beta D=0, \quad \forall D\in {\rm Der}(\mu)\,.
\end{equation}
The equality in \eqref{rest_2} holds if and only if $\beta+\lVert\beta\rVert^2I\in{\rm Der}(\mu)$.
\end{theorem}

\medskip
\begin{rem}\rm Note that condition \eqref{cond_rest} is always satisfied by some element in the ${\rm O}(n)$-orbit of $\mu$. If condition \eqref{cond_rest} is satisfied and $\mu\in S_\beta$, then
$$\beta=\text{\rm mcc}\{\alpha_{ij}^k:\mu_{ij}^k\neq0\}\,.$$
Here with $\text{\rm mcc}(X)$ we mean the unique element of minimal norm in the convex hull CH(X) of a subset $X\subset \mathfrak{t}$ (\cite{Lau10}).
\end{rem}

\section{Structure of solitons on Lie groups}
Let $(G,g)$ be a Lie group equipped with a left-invariant metric. Let $(\g,\lb)$ be the Lie algebra of $G$ and $\ip$ the inner product induced by $g$ on $\g$. Let 
$$\g=\r\oplus \n$$
be the orthogonal decomposition of $\g$, where $\n$ is the nilradical of $\g$, and
$$ \lambda:=[\cdot,\cdot]|_{\r\times\r}\,,\quad \sigma:=[\cdot,\cdot]|_{\r\times\n}\,,\quad \mu:=[\cdot,\cdot]|_{\n\times\n}\,.$$
Note that $\lambda$ can be further decomposed in $\lambda_0: \r\times\r\rightarrow \r$ and $\lambda_1: \r\times\r\rightarrow \n$.

\noindent Let $\beta$ such that $\mu \in \mathcal{S}_\beta$ and define $E_{\beta}\in {\rm End}(\g)$ by
$$E_\beta|_\r=0\,,\quad E_\beta|_\n= \beta+\lVert\beta\rVert^2I\,,$$ 
where $I$ is the identity of $\n$. Moreover, we denote by $\mm_\n:\n\to\n$ the endomorphism of $\n$ defined by using \eqref{M_tens} and, when $\r$ is a subalgebra of $\g$, we denote by  $\mm_\r:\r\to\r$ the endomorphism of $\r$. We have the following lemma.

\begin{lemma}\label{lem_E}\cite{LafLauJDG} Assume that $(\n,\mu)$ satisfies \eqref{cond_rest}. Then, $$\langle \pi(E_\beta)[\cdot,\cdot],[\cdot,\cdot]\rangle\geq 0$$ and 
$$\begin{aligned}
\langle \pi(E_\beta)[\cdot,\cdot],[\cdot,\cdot]\rangle =&\langle \pi(\beta+\lVert\beta\rVert^2I)\mu,\mu\rangle\\
& +\sum \langle (\beta+\lVert\beta\rVert^2I)[r_i,r_j],[r_i,r_j]\rangle\\
& +\sum 2\langle [\beta,\ad_{r_i}|_\n],\ad_{r_i}|_\n\rangle\,,
\end{aligned}$$
with $\{r_i\}$ orthonormal basis of $\r$. Moreover, each term is non-negative.
\end{lemma}

Henceforth, when confusion cannot occur, we identify tensor ${\sf K}$ with its associated endomorphism ${\sf K}_{g}$. Also ${\sf K}$-tensor components will be identify with their associated endomorphisms. The following lemma (whose proof is a direct computation) will be useful in the sequel.

\begin{lemma}\label{lem_M} Assume $[\r,\r]\subset \r$. Then, for any $A,B\in\r$ and $Z,W\in\n$,
$$\begin{aligned}
\langle {\rm M} Z, W\rangle =& \langle {\rm M}_\n Z,W\rangle+\frac 12\sum \langle [{\rm ad}_{r_i}|_\n,{\rm ad}_{r_i}^t|_\n]Z,W\rangle\,,\\
\langle {\rm M} A, B\rangle =& \langle {\rm M}_\r A,B\rangle-\frac 12 {\rm tr}({\rm ad}_A|_\n {\rm ad}_B^t |_\n)\,,\\
\langle {\rm M} A, W\rangle =&-\frac 12 {\rm tr}({\rm ad}_A|_\n {\rm ad}_W^t|_\n)\,,
\end{aligned}$$
where $\{r_i\}$ is an orthonormal basis of $\r$.
\end{lemma}

\begin{rem}\rm Note that under the assumptions of Lemma \ref{lem_M}, in matrix notation we have
\begin{equation}\label{matrix_M}
{\rm M}_g=\frac 12 \begin{bmatrix} 2{\rm M}_\r- \tilde B & - \tilde B\\
-\tilde B & 2{\rm M}_\n+ \sum [{\rm ad}_{r_i}|_\n,{\rm ad}_{r_i}^t|_\n]
\end{bmatrix}\,,
\end{equation}
where $\tilde B$ is the operator given by $\langle \tilde B X,Y\rangle = {\rm tr}({\rm ad}_X|_\n {\rm ad}_Y^t|_\n)$, for all $X,Y\in \g$, and the blocks are in terms of $\g = \r \oplus \n$.
\end{rem}

From now on we assume that the metric $g$ satisfies the semi-algebraic expanding soliton equation
$${\sf K}_g = c\,I+\frac 12(D+D^t), \qquad c <0, \quad D \in {\rm Der}(\g)\,,$$
and we set 
$$F:= S(\ad_H+D)\,,$$
where $S(A)$ is the symmetrization of $A\in{\rm End}(\g)$.

\begin{lemma} We have
\begin{equation}\label{lemma1} c\,{\rm tr}\,F+{\rm tr}\,F^2=0\,.
\end{equation}
\end{lemma}

\begin{proof} Let $E:=\ad_H+D$, then $E\in{\rm Der}(\g)$ and
$$\tr (c\, I- \hat Q + F)E = \tr \mm_gE = \frac 14\la \pi(E)[\,,\,],[\,,\,]\ra=0\,,$$
from \eqref{eqn_mm}. Since $\hat Q$ is invariant under automorphisms of $\g$, it follows $$e^{-t\tilde D^t}\hat Q e^{-t\tilde D}=\hat Q\,,$$ for any derivation $\tilde D\in{\rm Der}(\g)$. Differentiating at $t=0$, we have $D^t \hat Q  + \hat Q D=0$, which implies 
$$0 = \tr (D^t\hat Q + \hat Q D)=2\tr  \hat QD\,,$$
and the claim follows.
\end{proof}

Now we have

\begin{prop}\label{prop_strut} The orthogonal complement $\mathfrak r$ of the nilradical $\n$ is a reductive Lie subalgebra of $\g$ and
$$\g = \r \ltimes \n\,.$$
\end{prop}

\begin{proof} Without loss of generality we can suppose that condition \eqref{cond_rest} holds, since the claim condition is preserved by the ${\rm O}(n)$-action on $(\n,\mu)$ (see \cite{LafLauJDG} for more details).
\medskip

To prove the statement, we study separately the case when either $\n$ is abelian or not. In the former case, i.e. $\mu=0$, let $E\in{\rm End}(\g)$ be given by $E|_\r=0$ and $E|_\n= I$. Since ${\rm tr}\,F={\rm tr}\,F|_\n$ (\cite{LafLauJDG}, Lemma 2.6), by \eqref{eqn_mm} we have 
\begin{equation}\label{proof1}
c\,n+{\rm tr}\,F=\tr (c\, I- \hat Q + F)E = \tr \mm_gE=\frac 14 |\lambda_1|^2\,,
\end{equation} 
where $n:={\rm dim}(\n)$. Clearly, if $n=0$ the claim follows. Otherwise, from \eqref{lemma1} and \eqref{proof1} we have
$$ c\,n+{\rm tr}\, F \geq 0\quad \text{and}\quad {\rm tr}\, F^2 \leq n^{-1}({\rm tr}\,F)^2\,,$$
which force $\lambda_1=0$, $F|_\r=0$ and $F|_\n = t\,I$, for some $t\geq 0$.
\medskip

Now assume $\n$ non-abelian and recall that \eqref{cond_rest} holds. Then, in view of Lemma \ref{lem_E} we have
\begin{equation}\label{proof2}
\begin{aligned}
\langle \pi(E_\beta)[\cdot,\cdot],[\cdot,\cdot]\rangle=&\langle \pi(E_\beta)\lambda_0,\lambda_0\rangle+\langle \pi(E_\beta)\lambda_1,\lambda_1\rangle\\
&+2\langle \pi(E_\beta)\sigma,\sigma\rangle+\langle \pi(E_\beta)\mu,\mu\rangle\geq0\,,
\end{aligned}\end{equation}
which implies
\begin{equation}\label{proof3}
c\,{\rm tr}\, E_\beta+{\rm tr}\,FE_\beta=\tr (c\, I- \hat Q + F)E_\beta = \tr \mm_gE_\beta\geq 0\,,
\end{equation}
since \eqref{eqn_mm} holds and ${\rm tr}\, \hat Q E_\beta=0$. Hence, the following equalities hold (since ${\rm tr}\,\beta=-1$):
$${\rm tr} \,E_\beta^2=\rVert\beta\lVert^2\,{\rm tr}\, E_\beta\quad\text{and}\quad {\rm tr}\,FE_\beta=\rVert\beta\lVert^2\,{\rm tr}\,F\,,$$
and using the above formulae we have
$${\rm tr}\,F^2{\rm tr}\,E_\beta^2\leq ({\rm tr}\,FE_\beta)^2(\leq {\rm tr}\,F^2{\rm tr}\,E_\beta^2)\,,$$
which implies
$$F=tE_\beta\,,\quad \text{for some }t\geq 0 \,.$$
Moreover, since \eqref{lemma1} and \eqref{proof2} hold, we have
$$c\,{\rm tr}\, E_\beta+{\rm tr}\,FE_\beta=0\,$$
and $\lambda_1=0$. Hence, the claim follows.
\end{proof}

From the proof of Proposition \ref{prop_strut} we can easily deduce the following result.

\begin{prop}\label{cor} Assume $\mu\neq 0$ and satisfying \eqref{cond_rest}. Then
\begin{itemize}
\item[i.] $[\beta,{\rm ad}_\r|_\n]=0$,\\
\item[ii.] $\beta+\lVert\beta\rVert^2 I\in{\rm Der}(\n)$,\\
\item[iii.] $F= t\,E_\beta$, where $t=\frac{\tr F|_\n}{-1+\lVert\beta^2\rVert\,{\rm dim}\, \n}$.
\end{itemize}
While, for $\mu=0$ it follows $F|_\r=0$ and $F|_\n=t\, I$, where $t=\frac{\tr F|_\n}{{\rm dim}\, \n}$. 
\end{prop}
\begin{proof} Items (i) and (ii) respectively follow from \eqref{cond_brac} and \eqref{rest_2}, since $\langle \pi(E_\beta)[\cdot,\cdot],[\cdot,\cdot]\rangle=0$. The other claims follow directly by the previous proof.
\end{proof}

\begin{rem}\label{rem_dec}\rm
Let $\a$ be the center of $\r$. In view of Proposition \ref{prop_strut}, $\r$ is a reductive Lie algebra and consequently it decomposes as
$$\r=\h\oplus \a\,,$$
where $\h:=\lambda(\r,\r)$ is a semisimple Lie algebra. Hence, we can write $\g$ as
$$\g=(\h\oplus\a)\ltimes_\theta \n\,,$$
where $\theta(X):={\rm ad}_X|_\n$, for all $X\in\r$. However, since $\a$ is an abelian subalgebra of $\g$, we can also write
$$\g= \h\ltimes_\theta(\a\ltimes_\theta\n)\,,$$
where $\theta(X):={\rm ad}_X|_{\a\oplus\n}$ and $\theta(X)A=0$, for any $X\in\h$ and $A\in\a$.
\end{rem}

With the notations of Proposition \ref{prop_strut} in mind, we have the following

\begin{lemma}\label{lem_ad} We have
\begin{equation*}\label{eq1} {\rm ad}_X^t|_\n\in {\rm Der}(\n)\,,\end{equation*}
 for any $X\in \r$, and
\begin{equation*}\label{eq2}\sum[{\rm ad}_{r_i}|_\n,{\rm ad}_{r_i}^t|_\n]=0\,,\end{equation*}
where $\{r_i\}$ is an orthonormal basis of $\r$.
\end{lemma}

\begin{proof} If $\n$ is abelian, i.e. $\mu=0$, then the claims trivially follow.
Let's assume $\mu\neq 0$ and satisfying \eqref{cond_rest}. It follows from Propositions \ref{prop_strut} and \ref{cor} that $F=t E_\beta$, for some $t\geq 0$. Since ${\rm tr}\, F|_\n^2 = {\rm tr}\,F^2$, we have 
$$t=-\frac c{\lVert\beta\rVert^2}\quad\text{and}\quad F|_\n= -c\,I-\frac c{\lVert\beta\rVert^2}\beta\,.$$
Thus, from Lemma \ref{lem_M} and ${{\sf K}}|_\n = c\,I+\frac 12 (D|_\n+D^t|_\n) $ it follows
\begin{equation}\label{sum_0}
{\rm M}_\n+\frac 12 \sum[{\rm ad}_{r_i}|_\n,{\rm ad}_{r_i}^t|_\n] + \frac c{\lVert\beta\rVert^2}\beta=0\,.
\end{equation}
By tracing the left-hand side of \eqref{sum_0} and taking into account ${\rm tr}\,\beta=-1$ we obtain
$$c=-\frac 14 \lVert\beta\rVert^2\lVert \mu\rVert^2\,.$$
Moreover, since $\pi$ is a Lie algebra morphism and $\pi({\rm ad}_X)^t=\pi({\rm ad}_X^t)$, for all $X\in \g$, we have
\begin{equation}\label{ad_tran}
\begin{aligned}
{\rm tr}\,{\rm M}_\n [{\rm ad}_{r_i}|_\n,{\rm ad}_{r_i}^t|_\n] &=\frac 14 \langle \pi ({\rm ad}_{r_i}|_\n)\pi ({\rm ad}_{r_i}^t|_\n)\mu,\mu \rangle\\
  &=\frac 14 \langle \pi ({\rm ad}_{r_i}^t|_\n)\mu,\pi ({\rm ad}_{r_i})^t|_\n\mu \rangle\\
  &= \frac 14 \lVert\pi({\rm ad}_{r_i}^t|_\n)\mu\rVert^2\,, 
\end{aligned}
\end{equation}
for any $r_i\in\{r_i\}$, and multiplying \eqref{sum_0} by $M_\n$
$$\begin{aligned}
0 =& {\rm tr}\,{\rm M}_\n^2+ \frac 18\sum \lVert\pi({\rm ad}_{r_i}^t|_\n)\mu\rVert^2 + \frac c{\lVert\beta\rVert^2}{\rm tr}\,{\rm M}_\n\beta\\
   =&\frac 18\sum \lVert\pi({\rm ad}_{r_i}^t|_\n)\mu\rVert^2+ \frac {\lVert\mu \rVert^2}4 \left(  \frac4 {\lVert\mu \rVert^2} \lVert \mm_\n\rVert^2- \langle \mm_\n,\beta\rangle\right)\,.
\end{aligned}$$
Then, by \eqref{cond_m} we have
$$\langle \mm_\n,\beta\rangle\leq \frac4 {\lVert\mu \rVert^2} \lVert \mm_\n\rVert^2$$
and
$$\sum \lVert\pi({\rm ad}_{r_i}^t|_\n)\mu\rVert^2=0\,,$$
which implies ${\rm ad}_{r_i}^t|_\n\in{\rm Der}(\n)$, for all $i$, and the first claim follows.

To prove the second claim it is enough to observe that ${\rm M}_\n$ and $\beta$ are orthogonal to any derivation of $\n$, and applying \eqref{sum_0}
$$\sum[{\rm ad}_{r_i}|_\n,{\rm ad}_{r_i}^t|_\n]=0\,.$$
\end{proof}
\begin{rem}\label{equiv_ad}\rm By \eqref{ad_tran}, given a metric Lie algebra $\g$,
$$\sum[{\rm ad}_{r_i}|_\n,{\rm ad}_{r_i}^t|_\n]=0\,, \quad \text{for any orthonormal basis $\{r_i\}$ of $\r$}\,,$$
implies
$${\rm ad}_X^t|_\n\in {\rm Der}(\n)\,,\quad \text{for any $X\in\r$}\,.$$
\end{rem}

\subsection{Proof of the main results}
The next proposition implies our Theorem \ref{main_theo}.

\begin{prop}\label{prop_lie} Let $(G,g)$ be a Lie group equipped with a left-invariant metric and $\g$ its Lie algebra. Let $\g=\r\oplus\n$ be the orthogonal decomposition of $\g$, where $\n$ is the nilradical of $\g$, and let $g_\n$ be the pull-back of $g$ to the Lie group $N$ of $\n$. Then, $g$ is an expanding semi-algebraic ${\sf K}$-soliton if and only if
\begin{itemize}
\item[\rm (i)] $\g=\r\ltimes\n$, with $\r$ reductive Lie subalgebra and $\n$ nilradical of $\g$;
\item[\rm (ii)] $g_\n$ is an expanding algebraic ${\sf K}$-soliton on N;
\item[\rm (iii)] $\sum[{\rm ad}_{r_i}|_\n,{\rm ad}_{r_i}^t|_\n]=0$, where $\{r_i\}$ is an orthonormal basis of $\r$;
\item[\rm (iv)] for any $X,Y\in \r$
$${\sf K}(g_\r)(X, Y)=c\, g_\r(X, Y)+\frac 1{2}\tr (\ad_X|_\n \ad_{ Y}^t|_\n)-\frac 12 \tr \ad_X\cdot\tr\ad_{ Y}\,,$$
where $g_\r$ is the pull-back of $g$ to the Lie group of $\r$.
\end{itemize}
\end{prop}

\begin{proof} Let $(G,g)$ be an expanding semi-algebraic ${\sf K}$-soliton with ${\sf K}_g= c\,I+\frac 12(D+D^t)$, for some $D\in {\rm Der}(\g)$, and denote with $\tilde B:\g\to\g$ the endomorphism defined by $$\langle \tilde B X,Y\rangle= \tr({\rm ad}_X|_\n {\rm ad}_Y^t|_\n)\,.$$ Items (i) and (iii) follow from Proposition \ref{prop_strut} and Lemma \ref{lem_ad}, respectively. Item (iv) follows from Proposition \ref{prop_strut} and Lemma \ref{lem_M}, since
$$\mm|_\r+\hat Q|_\r= c\,I|_\r \quad \text{and}\quad \mm_\r = \mm|_\r+\frac 12 \tilde B|_\r\,.$$
Finally, item (ii) follows from Lemma \ref{lem_M} and Lemma \ref{lem_ad}. Indeed, 
$$(c\,I+S(D))|_\n=\mm|_\n-S(\ad_H)|_\n=\mm_\n-S(\ad_H)|_\n={\sf K}_{g_\n}-S(\ad_H)|_\n\,,$$
where ${\sf K}_{g_\n}$ denotes the ${\sf K}$-operator of the Lie algebra $\n$. Thus, the claim follows and it turns out that the derivation associated to $g_\n$ is given by $D_1 = S(\ad_H+D)|_\n$.

\medskip
Viceversa, suppose that (i)-(iv) hold. Let $\n=\n_1\oplus\ldots\oplus\n_r$ be an orthogonal decomposition of $\n$ such that
$$[\n,\n]= \n_2\oplus\ldots\oplus\n_r \,, \quad [\n,[\n,\n]]=\n_3\oplus\ldots\oplus\n_r$$
and so on. Since $\ad_X|_\n$ and $\ad_X^t|_\n$ are both derivations by Remark \ref{equiv_ad}, we have ${\rm ad}_X(\n_i)\subset\n_i$ and ${\rm ad}_Z(\n_i)\subset\n_{i+1}$, for any $X\in \r$ and $Z\in\n$. Thanks to Lemma \ref{lem_M} and (iii), under these assumptions, we have
$$
{\rm M}=\begin{bmatrix} {\rm M}_\r- \frac 12\tilde B & 0\\
0 & {\rm M}_\n
\end{bmatrix} \quad \text{and}\quad {\sf K}=\begin{bmatrix} \ast & 0\\ 0 & \ast
\end{bmatrix}\,,$$
where the block representations are with respect to $\g=\r\oplus\n$. 

Now, let $D_1$ be the derivation characterizing $g_{\n}$ and $D:= -\ad_H+\left[ \begin{smallmatrix} 0 & 0 \\ 0 & D_1\end{smallmatrix}\right]$. Since $\r$ is reductive and (iv) holds, we have
$${\sf K}|_\r= \mm|_\r -S(\ad_H)|_\r+ \hat Q|_\r=\mm_\r-\frac 12 \tilde B|_\r-S(\ad_H)|_\r+ \hat Q|_\r\,,$$
which implies ${\sf K}|_\r = c\,I-S(\ad_H)$. Similarly,
$${\sf K}|_\n=\mm|_\n-S(\ad_H)|_\n=\mm_\n-S(\ad_H)|_\n$$
and ${\sf K}|_\n = c\,I + S(-\ad_H+D_1)$, since (ii) holds.

\medskip
It only remains to show that $D\in {\rm Der}(\g)$. To prove the claim it is enough that $\left[ \begin{smallmatrix} 0 & 0 \\ 0 & D_1\end{smallmatrix}\right]\in{\rm Der}(\g)$, or equivalently $[D_1,\ad_X|_\n]=0$, for any $X\in\r$. However, since ${\sf K}_{g_\n} = \mm_\n= c\,I+D_1$ and $\mm_\n$ commutes with any derivation of $\n$ whose transpose is also a derivation (\cite{LafLauJDG}, Remark 2.5), the claim follows.
\end{proof}

Corollary \ref{main_theo1} follows since in the solvable case $\r$ is abelian and, consequently, $K(g_\r)=0$.
\begin{rem}\rm When $G$ is unimodular the derivation $D:= {\sf K}_g-c I$ only acts on the nilradical $\n$ of $\g$, since $H=0$ (and therefore $\ad_H=0$) by definition.
\end{rem}

\section{Applications}
In this section we use our results to construct explicit examples of expanding algebraic solitons to HCF on complex Lie groups. 

We work on 4-dimensional solvable (non-nilpotent) complex unimodular Lie algebras, which are classified by the following list (see e.g. \cite{Burde}): 

\begin{itemize}
\item $\s_{3,-1}\oplus \C$, with structure equations
$$
[Z_1,Z_2]=Z_2\,,\quad [Z_1,Z_3]=-Z_3\,;
$$

\item $\g_1(-2)$, with structure equations
$$
[Z_1,Z_2]=Z_2\,,\quad [Z_1,Z_3]=Z_3\,,\quad [Z_1,Z_4]=-2\, Z_4\,;
$$
\item $\g_4$, with structure equations
$$
[Z_1,Z_2]=Z_3\,,\quad [Z_1,Z_3]=Z_4\,,\quad [Z_1,Z_4]=Z_2\,;
$$
\item $\g_7$, with structure equations
$$
[Z_1,Z_2]=Z_3\,,\quad [Z_1,Z_3]=Z_2\,,\quad [Z_2,Z_3]=Z_4\,;
$$
\item $\g_3(\alpha)$, with structure equations
$$
[Z_1,Z_2]=Z_3\,,\quad [Z_1,Z_3]=Z_4\,,\quad [Z_1,Z_4]=\alpha(Z_2+Z_3)\,, \quad \alpha\in\C^\ast\,. 
$$
\end{itemize}

We show that in the first four cases ($\s_{3,-1}\oplus \C, \g_1(-2), \g_4, \g_7$) there exists a soliton to HCF on the corresponding Lie group (in view of \cite[Theorem 1.2]{LPV}, a complex unimodular Lie groups has at most one algebraic soliton to HCF up to homotheties). In the last case the existence of a soliton remains an open question.    

\subsection{$\s_{3,-1}\oplus\C$} Let  $g$ be a Hermitian inner product on $\s_{3,-1}\oplus\C$. We can find a $g$-unitary basis $\{W_i\}$ such that
$$
W_1\in \langle Z_1,Z_2,Z_3,Z_4\rangle\,,\quad  W_2\in \langle Z_2,Z_3,Z_4\rangle\,,\quad W_3\in \langle Z_3,Z_4\rangle\,,\quad W_4\in \langle Z_4\rangle \,.
$$
With respect to this new basis, we have
$$
[W_1,W_2]=pW_2+qW_3+rW_4\,,\quad [W_1,W_3]=-pW_3+sW_4\,,
$$
for some $p,q,r,s\in\C$ with $p\neq 0$, and $$\s_{3,-1}\oplus\C=\r\oplus\n\,,$$
where $\r=\la W_1\ra$ and $\n=\la W_2,W_3,W_4\ra$. 

Since the nilradical $\n$ is an abelian ideal, $g_\n$ trivially induces an expanding algebraic soliton to HCF on the Lie group of $\n$. Therefore, by Corollary \ref{main_theo1}, $g$ induces an expanding algebraic soliton to HCF on the Lie group of $\s_{3,-1}\oplus\C$ if and only if $$[\ad_{W_1}|_\n,\ad_{\bar W_1}^t|_\n]=0\quad \text{and}\quad g(W_1,\bar W_1)= -\frac1{2c} \tr(\ad_{W_1}|_\n\ad_{\bar W_1}^t|_\n)\,.$$
It is straightforward to show that the first condition holds if and only if $q,r,s=0$; while, since $\{W_i\}$ is a $g$-unitary basis, we have
$$1=g(W_1,\bar W_1)= -\frac1{2c} \tr(\ad_{W_1}|_\n\ad_{\bar W_1}^t|_\n)=-\frac{|p|^2}c\,,$$
which implies $c=-|p|^2$. Thus in matrix notation, with respect to $\{W_i\}$, we have
$$K_g=-|p|^2 I + D\,,$$
where $D:={\rm diag}(0, |p|^2, |p|^2, |p|^2)$.

Finally, we note that 
$$g(Z_2,\bar Z_3)=g(Z_2,\bar Z_4)=g(Z_3,\bar Z_4)=0 \iff q=r=s=0\,,$$
and we have  following result:
\begin{prop} A Hermitian inner product $g$ on $\s_{3,-1}\oplus\C$ induces an expanding algebraic soliton to HCF on the corresponding (simply connected) Lie group if and only if $g(Z_2,\bar Z_3)=g(Z_2,\bar Z_4)=g(Z_3,\bar Z_4)=0$.
\end{prop}

\subsection{$\g_1(-2)$} Given a Hermitian inner product $g$ on $\g_1(-2)$, there exists a $g$-unitary basis satisfying
$$
W_1\in \langle Z_1,Z_2,Z_3,Z_4\rangle\,,\quad  W_2\in \langle Z_2,Z_3,Z_4\rangle\,,\quad W_3\in \langle Z_3,Z_4\rangle\,,\quad W_4\in \langle Z_4\rangle \,. 
$$
With respect to this new basis, we have
$$
[W_1,W_2]=pW_2+qW_3+rW_4\,,\quad [W_1,W_3]=sW_3+tW_4\,,\quad [W_1,W_4]=uW_4\,,
$$
for some $p,q,r,s,t,u\in\C$, where $p+s+u=0$ and $p,s,u\neq0$. Then, $\g_1(-2)$ splits in $$\g_1(-2)=\r\oplus\n\,,$$
where $\r=\la W_1\ra$ and $\n=\la W_2,W_3,W_4\ra$, and $g_\n$ gives rise to an expanding algebraic soliton to HCF on the Lie group of $\n$, since $\n$ is an abelian ideal. 

Now, a direct computation yields that 
$$
[\ad_{W_1}|_\n,\ad_{\bar W_1}^t|_\n]=0 \quad \text{if and only if}\quad q,r,t=0\,;
$$
while
$$1=g(W_1,\bar W_1)= -\frac1{2c} \tr(\ad_{W_1}|_\n\ad_{\bar W_1}^t|_\n)=-\frac{|p|^2+|s|^2+|u|^2}{2c}\,,$$
since $\{W_i\}$ is a $g$-unitary basis. Therefore, if $q,r,t=0$ and $c=-(|p|^2+|s|^2+|u|^2)/2$, the assumptions in Corollary \ref{main_theo1} are satisfied and, in matrix notation with respect to $\{W_i\}$, we have 
$$
K_g=c I + D\,,
$$ 
where $D:=-{\rm diag}(0, c, c, c)$.

Noting that
$$g(Z_2,\bar Z_3)=g(Z_2,\bar Z_4)=g(Z_3,\bar Z_4)=0 \iff q=r=t=0\,,$$
we obtain the following result.
\begin{prop} A Hermitian inner product $g$ on $\g_1(-2)$ induces an expanding algebraic soliton to HCF on the corresponding (simply connected) Lie group if and only if $g(Z_2,\bar Z_3)=g(Z_2,\bar Z_4)=g(Z_3,\bar Z_4)=0$.
\end{prop}

\subsection{$\g_4$} Let $\tilde g$ be a Hermitian inner product on $\g_4$ such that $Z_2,Z_3,Z_4$ are orthogonal to each other. Let $\{W_i\}$ be a  $\tilde g$-unitary basis satisfying 
$$
W_1\in \langle Z_1,Z_2,Z_3,Z_4\rangle\,,\quad  W_2\in \langle Z_2\rangle\,,\quad W_3\in \langle Z_3\rangle\,,\quad W_4\in \langle Z_4\rangle \,. 
$$
Then, we have
$$
[W_1,W_2]=pW_3\,,\quad [W_1,W_3]=qW_4\,,\quad [W_1,W_4]=rW_2\,,
$$
and we assume $p,q,r\in\R^+\backslash \{0\}$. Hence, $\g_4$ splits as
$$\g_4=\r\oplus\n\,,$$
where $\r=\la W_1\ra$ and $\n=\la W_2,W_3,W_4\ra$. 

Since $\n$ is an abelian ideal, $\tilde g_\n$ induces an expanding algebraic soliton to HCF on the Lie group of $\n$. Moreover, by Corollary \ref{main_theo1}, $\tilde g$ induces an expanding algebraic soliton to HCF on the Lie group of $\g_4$ if and only if $$[\ad_{W_1}|_\n,\ad_{\bar W_1}^t|_\n]=0\quad \text{and}\quad 1=g(W_1,\bar W_1)= -\frac1{2c} \tr(\ad_{W_1}|_\n\ad_{\bar W_1}^t|_\n)\,.$$
The first condition is equivalent to require $p=q=r$, while the second one is satisfied if and only if $c=-\frac 32 p^2$. Hence, in matrix notation with respect to $\{W_i\}$, we obtain
$$K_{\tilde g}=-\frac 32 p^2 I + D\,,$$ 
where $D:=\frac 32{\rm diag}(0, p^2, p^2, p^2)$.

Finally, we note that 
$$\tilde g(Z_2,\bar Z_2)=\tilde g(Z_3,\bar Z_3)=\tilde g(Z_4,\bar Z_4) \iff p=q=r\,,$$
and we have  following result:

\begin{prop} A Hermitian inner product on $\g_4$ induces an expanding algebraic soliton to HCF on the corresponding (simply connected) Lie group if and only if it is homothetically equivalent to a Hermitian inner product $ g$ on $\g_4$ satisfying $ g(Z_2,\bar Z_2)= g(Z_3,\bar Z_3)= g(Z_4,\bar Z_4)$ and $ g(Z_2,\bar Z_3)= g(Z_2,\bar Z_4)= g(Z_3,\bar Z_4)=0$.
\end{prop}

\subsection{$\g_7$} Let  $\tilde g$ be the standard Hermitian inner product on $\g_7$. Then, $\g_7$ splits in 
$$\g_7=\r\oplus \n\,,$$
where $\r=\la Z_1\ra$ and $\n=\la Z_2,Z_3,Z_4\ra$ is isomorphic to $\h_3(\C)$, the Lie algebra of the 3-dimensional complex Heisenberg Lie group $\mathbb H_3(\C)$. 

In view of \cite[Proposition 4.1]{LPV}, any left-invariant Hermitian metric on $\mathbb H_3(\C)$ is an expanding soliton to HCF.  Therefore $\tilde g_\n$ induces an expanding algebraic soliton to HCF on the Lie group of $\n$, and a straightforward computation yields that 
$$[\ad_{Z_1}|_\n,\ad_{\bar Z_1}^t|_\n]=0 \quad \text{and}\quad \tr(\ad_{Z_1}|_\n\ad_{\bar Z_1}^t|_\n)=2\,.$$
Then, the assumptions in Corollary \ref{main_theo1} are satisfied if and only if $c=-1$, and in such a case we have
$$
K_{\tilde g}=- I + D\,,
$$
where $D:={\rm diag}(0, 1,1,1)$. Hence, we can claim the following proposition:
\begin{prop} A Hermitian inner product on $\g_7$ induces an expanding algebraic soliton to HCF on the corresponding (simply connected) Lie group if and only if it is homothetically equivalent to $\tilde g$.
\end{prop}

\end{document}